\documentclass[reqno,12pt,oneside]{amsart}

\usepackage{amsthm,amsmath,amssymb,amsfonts}
\usepackage[colorlinks=true,hypertexnames=false,linkcolor=blue,citecolor=blue]{hyperref}
\usepackage[top=3cm,bottom=2.5cm,left=2.5cm, right=3cm,marginparwidth=0cm]{geometry}

\usepackage{epsfig,graphics}
\usepackage{amscd}
\usepackage{amsmath}
\usepackage{mathtools}
\usepackage{bm}
\usepackage{amssymb}
\usepackage{graphicx}
\usepackage{psfrag}
\usepackage{verbatim}
\usepackage{tikz}
\usetikzlibrary{quotes}
\usetikzlibrary{topaths}
\usetikzlibrary{shapes,arrows,cd}

\usepackage{adjustbox}

\newtheorem{theorem}{Theorem}[section]
\newtheorem{lemma}[theorem]{Lemma}
\newtheorem{prop}[theorem]{Proposition}
\newtheorem{coro}[theorem]{Corollary}

\theoremstyle{definition}
\newtheorem{definition}[theorem]{Definition}

\theoremstyle{remark}
\newtheorem{remark}[theorem]{Remark}

\theoremstyle{plain}
\newtheorem{innercustomthm}{Cross-Ratio Inequality\!}
\newenvironment{customthm}[1]
  {\renewcommand\theinnercustomthm{#1}\innercustomthm}
  {\endinnercustomthm}

\theoremstyle{plain}
\newtheorem{maintheorem}{Theorem}

\newcommand{\Z}{\ensuremath{\mathbb{Z}}}

\newcommand{\nt}{\ensuremath{\mathbb{N}}}
\newcommand{\car}{\operatorname{Card}}
\newcommand{\crit}{\operatorname{Crit}}

\newcommand{\crd}{\operatorname{CrD}}

\begin{document}

\title[]{There are no $\sigma$-finite absolutely continuous invariant measures for multicritical circle maps}

\author{Edson de Faria}
\address{Instituto de Matem\'atica e Estat\'istica, Universidade de S\~ao Paulo}
\curraddr{Rua do Mat\~ao 1010, 05508-090, S\~ao Paulo SP, Brasil}
\email{edson@ime.usp.br}

\author{Pablo Guarino}
\address{Instituto de Matem\'atica e Estat\'istica, Universidade Federal Fluminense}
\curraddr{Rua Prof. Marcos Waldemar de Freitas Reis, S/N, 24.210-201, Bloco H, Campus do Gragoat\'a, Niter\'oi, Rio de Janeiro RJ, Brasil}
\email{pablo\_\,guarino@id.uff.br}

\thanks{This work has been supported by ``Projeto Tem\'atico Din\^amica em Baixas Dimens\~oes'' FAPESP Grant  2016/25053-8, and also by Coordena\c{c}\~ao de Aperfei\c{c}oamento de Pessoal de N\'ivel Superior - Brasil (CAPES) grant 23038.009189/2013-05.}


\begin{abstract} It is well-known that every multicritical circle map without periodic orbits admits a unique invariant Borel {\it probability\/} measure which is {\it purely singular\/} with respect to Lebesgue measure. Can such a map leave invariant an {\it infinite, $\sigma$-finite\/} invariant measure which is absolutely continuous with respect to Lebesgue measure? 
In this paper, using an old criterion due to Katznelson, we show that the answer to this question is no. 
\end{abstract}

\maketitle

\vspace{-0.5cm}

\section{Introduction}

In this paper we study certain ergodic-theoretic properties of {\it multicritical circle maps\/} -- orientation-preserving homeomorphisms of the circle that are reasonably smooth and have a finite number of critical points, all of which are non-flat of power-law type. 

It is well-known that a multicritical circle map $f: S^1\to S^1$ without periodic points is minimal and uniquely ergodic. Its unique invariant Borel probability measure turns out to be singular with respect to Lebesgue measure $\lambda$ on $S^1$ (see Section \ref{secprel} for precise references). At least in principle, this fact does not rule out the possibility that $f$ leaves invariant an {\it infinite,  $\sigma$-finite\/} measure which is absolutely continuous with respect to Lebesgue measure. If such a measure $\mu$ exists, and we denote by $\psi=d\mu/d\lambda$ its Radon-Nikodym derivative with respect to Lebesgue, then $\psi$ is a Borel function such that $0<\psi<\infty$ Lebesgue-a.e., and we have
the cocycle identity 
\begin{equation}\label{radon}
  \psi(x) = 
  \psi\circ f(x)\cdot Df(x)\ \ \ \textrm{for Lebesgue a.e.}\ x\in S^1
  \end{equation}                                                                                                                                                                                                                                                                                                                                                                                                                                                                                                                                                                                                                                                                            
                                                                                                    One can ask more generally: When does a minimal $C^1$  homeomorphism of the circle admit an infinite $\sigma$-finite invariant measure which is absolutely continuous with respect to Lebesgue measure? As it turns out, there are indeed examples of $C^\infty$ {\it diffeomorphisms\/} of the circle with this property, as shown by Katznelson in \cite{katz1}. However, as we will prove below, there are no such examples in the realm of multicritical circle maps. Our main theorem can thus be stated as follows.

\begin{maintheorem}\label{main} If $f:S^1\to S^1$ is a $C^3$ multicritical circle map without periodic points, then $f$ admits no $\sigma$-finite invariant measure which is absolutely continuous with respect to Lebesgue measure. 
\end{maintheorem}

The proof of this result (to be given in Section \ref{secproofmain}) will comprise two separate arguments. The first argument will prove the statement for {\it almost all\/} irrational rotation numbers only: a certain subset of the set of rotation numbers of bounded type will be excluded. The second argument will prove the statement for all bounded type rotation numbers. In both cases, the \emph{Schwarzian derivative} of $f$ is used in a fundamental way, which is why we restrict our attention to $C^3$ dynamics. However, it is quite possible that the statement of Theorem \ref{main} holds true under less regularity (perhaps $C^{2+\alpha}$ smoothness is enough).

\subsection*{Brief summary} Here is how the paper is organized.
In the preliminary Section \ref{secprel}, we present the basic facts about multicritical circle maps and recall the fundamental tools: the real bounds, the cross-ratio inequality, Koebe's distortion principle, Yoccoz's inequality. In  Section \ref{seckatz}, we establish a criterion for \emph{non-existence} of $\sigma$-finite absolutely continuous invariant measures. Since this is a slight generalization of \cite[Th.~1]{katz1}, we call it the \emph{Katznelson criterion}. In Section \ref{secproofmain}, we use Katznelson's criterion to prove two particular versions of Theorem \ref{main}, namely Theorem \ref{nosigcrit} and Theorem \ref{nosigbound}. The former deals with all unbounded type rotation numbers and most bounded type ones, and its proof uses Yoccoz's inequality. The latter deals exclusively with bounded type rotation numbers, and its proof depends on a negative Schwarzian property of first return maps whose proof is given in  Appendix~\ref{app}. Combining Theorems \ref{nosigcrit} and \ref{nosigbound}, we immediately deduce Theorem \ref{main}. 

\section{Preliminaries}\label{secprel}

The non-wandering set $\Omega(f)$ of a circle homeomorphism $f$ without periodic points can be either the whole circle -- in which case we say that $f$ is \emph{minimal} -- or else a Cantor set. In the latter case, we say that $f$ is a \emph{Denjoy counterexample}, or that $\Omega(f)$ is an \emph{exceptional minimal set}. In both cases, the rotation number of $f$ is necessarily irrational. 

In his classical article \cite{denjoy}, Denjoy
constructed circle diffeomorphisms (of class $C^{1+\alpha}$ for some $\alpha>0$) having an arbitrary irrational rotation number and possessing an exceptional minimal set. For any such diffeomorphism $f$, even when its minimal set has zero Lebesgue measure, it is easy to construct an $f$-invariant $\sigma$-finite measure which is absolutely continuous with respect to Lebesgue. Indeed, it is enough to consider Lebesgue measure on any interval $I$ in the complement of $\Omega(f)$, and then spread this measure by $f$ to the whole orbit of $I$, namely $\big\{f^n(I)\big\}_{n\in\Z}$. This produces an $f$-invariant $\sigma$-finite measure (definitely \emph{not} finite), which is absolutely continuous since $f$, being smooth, preserves sets of zero Lebesgue measure. 

One might be tempted to think that such $\sigma$-finite, absolutely continuous invariant measures can only be constructed when the diffeomorphism $f$ has a wandering interval (such as $I$ above), but in \cite{katz1} Katznelson constructed  \emph{minimal} $C^{\infty}$ diffeomorphisms (with very special rotation numbers) which \emph{do} admit such invariant measures. 

In the context of circle maps with \emph{critical points}, we recall that Hall was able to construct in \cite{hall} (see also \cite{liviana}) $C^{\infty}$ circle homeomorphisms which are Denjoy counterexamples. Hence the same construction explained above can be performed here in order to produce invariant measures which are $\sigma$-finite and absolutely continuous with respect to Lebesgue. We remark that the critical points of maps studied in both \cite{hall} and \cite{liviana} satisfy some \emph{flatness} condition. 

The main result of our paper, namely Theorem \ref{main}, states that there are no such examples amongst smooth circle homeomorphisms whose critical points satisfy the following non-flatness condition.

\begin{definition}\label{defmccm} A critical point $c$ of a one-dimensional $C^3$ map $f$ is said to be \emph{non-flat} of criticality $d>1$ if there exists a neighbourhood $W$ of $c$ such that $f(x)=f(c)+\phi(x)\big|\phi(x)\big|^{d-1}$ for all $x \in W$, where $\phi : W \rightarrow \phi(W)$ is a $C^{3}$ diffeomorphism satisfying $\phi(c)=0$. A \emph{multicritical circle map} is an orientation preserving $C^3$ circle homeomorphism $f$ having 
$N \geq 1$ critical points, all of which are non-flat.
\end{definition}

Being a homeomorphism, a multicritical circle map $f$ has a well defined rotation number. We will focus on the case when this number is irrational, which is equivalent to saying that $f$ has no periodic orbits. In particular, $f$ is \emph{uniquely ergodic}: it preserves a unique Borel probability measure $\mu$. Furthermore, we have the following fundamental result due to J.-C. Yoccoz \cite{yoccoz1}.

\begin{theorem}\label{yoccoztheorem} Let $f$ be a multicritical circle map with irrational rotation number $\rho$. Then $f$ is topologically conjugate to the rigid rotation $R_{\rho}$, i.e., there exists a homeomorphism $h: S^{1} \rightarrow S^{1}$ such that $h \circ f = R_{\rho} \circ h.$
\end{theorem}

Therefore, the unique $f$-invariant probability measure $\mu$ is just the push-forward of the Lebesgue measure under $h^{-1}$, that is, $\mu(A)=\lambda\big(h(A)\big)$ for any Borel set $A$, where $\lambda$ denotes the normalized Lebesgue measure in the unit circle (recall that the conjugacy $h$ is unique up to post-composition with rotations, so the measure $\mu$ is well-defined). In other words, the following diagram commutes.
$$
\begin{CD}
(S^1,\mu)@>{f}>>(S^1,\mu)\\
@V{h}VV             @VV{h}V\\
{(S^1,\lambda)}@>>{R_{\rho}}>{(S^1,\lambda)}
\end{CD}
$$
Note, in particular, that $\mu$ has no atoms and gives positive measure to any non-empty open set. However, as already mentioned in the introduction, $\mu$ is never absolutely continuous with respect to Lebesgue. More precisely:

\begin{theorem}\label{singular} Let $f$ be a multicritical circle map with irrational rotation number. Then its unique invariant probability measure is purely singular with respect to Lebesgue measure.
\end{theorem}

This theorem was proved by Khanin in the late eighties, by means of a certain thermodynamic formalism \cite[Theorem 4]{khanin91} (see also \cite[Proposition 1]{GS93}). We would like to point out that Theorem \ref{singular} is a straightforward consequence of our main result, namely Theorem \ref{main}, as it follows from the simple observation that either $\mu$ is absolutely continuous with respect to Lebesgue, or else it is singular. Otherwise we would have a decomposition $\mu=\nu_1+\nu_2$, where $\nu_1$ is absolutely continuous, $\nu_2$ is singular and both are non-zero. Since $f$ preserves sets of zero Lebesgue measure, both $\nu_1$ and $\nu_2$ would be $f$-invariant, contradicting the unique ergodicity of $f$. Since by Theorem \ref{main} $f$ admits no invariant measure which is absolutely continuous (neither finite nor $\sigma$-finite), Theorem \ref{singular} follows. For more on the ergodic theory of multicritical circle maps, see \cite{dFG2016}.

\subsection{The real bounds} As it is well known, any irrational number $\rho\in (0,1)$ has an infinite \emph{continued fraction expansion}, say  
\begin{equation*}
      \rho(f)= [a_{0} , a_{1} , \cdots ]=
      \cfrac{1}{a_{0}+\cfrac{1}{a_{1}+\cfrac{1}{ \ddots} }} \ .
\end{equation*}
The coefficients $a_n$ are called the \emph{partial quotients} of $\rho$.
Truncating this expansion at level $n-1$, we obtain a sequence of irreducible fractions $p_n/q_n=[a_0,a_1, \cdots ,a_{n-1}]$, which are called the \emph{convergents} of the irrational $\rho$.
The sequence of denominators $q_n$, which we call the \emph{return times}, satisfies
\begin{equation*}
 q_{0}=1, \hspace{0.4cm} q_{1}=a_{0}, \hspace{0.4cm} q_{n+1}=a_{n}\,q_{n}+q_{n-1} \hspace{0.3cm} \text{for $n \geq 1$} .
\end{equation*}

Now let $f$ be a circle homeomorphism with rotation number $\rho(f)=\rho$. 
For any given $x \in S^1$ we construct a nested sequence of partitions of the circle $\big\{\mathcal{P}_n(x)\big\}_{n\in\nt}$ as follows: for each non-negative integer $n$, let $I_{n}(x)$ be the interval with endpoints $x$ and $f^{q_n}(x)$ containing  $f^{q_{n+2}}(x)$, namely,
$I_n(x)=\big[x,f^{q_n}(x)\big]$ and $I_{n+1}(x)=\big[f^{q_{n+1}}(x),x\big]$. We write $I_{n}^{j}(x)=f^{j}\big(I_{n}(x)\big)$ for all $j$ and $n$. It is well known that, for each $n\geq 0$, the collection of intervals$$\mathcal{P}_n(x)\;=\; \big\{I_n^i:\;0\leq i\leq q_{n+1}-1\big\}\cup\big\{I_{n+1}^j:\;0\leq j\leq q_{n}-1\big\}$$is a \emph{partition of the circle modulo endpoints} (see for instance \cite[Lemma 2.4]{EdFG18}), called the {\it $n$-th dynamical partition\/} associated to $x$. The intervals of the form $I_n^i$ are called \emph{long}, whereas those of the form $I_{n+1}^j$ are called \emph{short}. The following fundamental result was obtained by Herman and \'Swi\c{a}tek in the late eighties \cite{H,swiatek}.

\begin{theorem}[Real bounds]\label{realbounds} Given $N\geq 1$ in $\nt$ and $d>1$ there exists a universal constant $C=C(N,d)>1$ with the following property: for any given multicritical circle map $f$ with irrational rotation number, and with at most $N$ critical points whose criticalities are bounded by $d$, there exists $n_0=n_0(f)\in\nt$ such that for each critical point $c$ of $f$, for all $n \geq n_0$, and for every pair $I,J$ of adjacent atoms of $\mathcal{P}_n(c)$ we have:$$C^{-1}\,|I| \leq |J| \leq C\,|I|\,,$$where $|I|$ denotes the Euclidean length of an interval $I$.
\end{theorem}

A detailed proof of Theorem \ref{realbounds} can also be found in \cite{EdF18,EdFG18}. In what follows, two positive real numbers $\alpha$ and $\beta$ are said to be {\it comparable modulo\/} $f$  (or simply {\it comparable\/}) if there exists a constant $K>1$, depending only on $f$, such that $K^{-1}\beta\leq \alpha\leq K\beta$. This relation is denoted $\alpha\asymp \beta$. Therefore, Theorem \ref{realbounds} states that $|I|\asymp|J|$ for any two adjacent atoms $I$ and $J$ of a dynamical partition associated to a critical point of $f$.

\subsection{Some geometric tools}\label{sectools} We finish Section \ref{secprel} reviewing some classical tools from one-dimensional dynamics, that will be used along the text. Given two intervals $M\subset T\subset S^{1}$, with $M$ compactly contained in $T$ (written $M\Subset T$), we denote by $L$ and $R$ the two connected components of $T\setminus M$. We define the \emph{space} of $M$ inside $T$ as the smallest of the ratios $|L|/|M|$ and $|R|/|M|$. If the space is $\tau>0$, we say that $T$ contains a \emph{$\tau$-scaled neighbourhood} of $M$.

\begin{lemma}[Koebe distortion principle]\label{koebe} For each $\ell,\tau>0$ and each multicritical circle map $f$ there exists a 
constant $K=K(\ell,\tau,f)>1$ of the form
$$K= \left(1 + \frac{1}{\tau} \right)^2 \exp (C_0\,\ell)\,,$$where $C_0$ is a constant depending only on $f$, with the following property. If\, $T$ is an interval such that $f^k|_{T}$ is a diffeomorphism onto its image, for some $k\in\nt$, and if $\sum_{j=0}^{k-1} |f^j(T)|\leq \ell$, then for each interval $M\subset T$ for which $f^k(T)$ contains a $\tau$-scaled neighbourhood of $f^k(M)$ one has
\[
\frac{1}{K}\leq \frac{|Df^k(x)|}{|Df^k(y)|}\leq K\quad\mbox{for all $x,y\in M$.}
\]
\end{lemma}

A proof of Koebe distortion principle can be found in \cite[Section IV.3, Theorem 3.1]{dMvS}. We define the \emph{cross-ratio} of the pair $M,T$ to be the ratio
\begin{equation*}
[M,T]= \frac{|L|\,|R|}{|L \cup M|\,|R \cup M|}\, \in (0,1).
\end{equation*}
The \textit{cross-ratio distortion} of a homeomorphism $f:S^{1}\to S^{1}$ on the pair $M,T$ is defined as
\begin{equation*}
\crd(f;M,T)= \frac{\big[f(M),f(T)\big]}{[M,T]}\,.
\end{equation*}
We have the following chain rule for the cross-ratio distortion:$$\crd(f^j;M,T) = \prod_{i=0}^{j-1} \crd\big(f;f^{i}(M), f^{i}(T)\big)\,.$$Given a family of intervals $\mathcal{F}$ on $S^{1}$ and a positive integer $m$, we say that $\mathcal{F}$ has {\it multiplicity of intersection at most $m$\/} if each $x\in S^{1}$ belongs to at most $m$ elements of $\mathcal{F}$. 

\begin{customthm}{} Given a multicritical critical circle map $f:S^1\to S^1$, there exists a constant $C>1$, depending only on $f$, 
such that the following holds. If $M_i\Subset T_{i} \subset S^1$, where $i$ runs through some finite set of indices $\mathcal{I}$, 
are intervals on the circle such that the family $\{T_i: i\in \mathcal{I}\}$ 
has multiplicity of intersection at most $m$, then$$\prod_{i \in \mathcal{I}} \crd(f;M_{i},T_{i}) \leq C^{m}.$$
\end{customthm}

The Cross-Ratio Inequality was obtained by \'Swi\c{a}tek in \cite{swiatek} (see also \cite[Theorem B]{EdFG18}). A sketch of the proof can be found in \cite[page 5589]{EdF18}. We remark that similar estimates were used before by Yoccoz \cite{yoccoz1}, on his way to proving Theorem \ref{yoccoztheorem} (see \cite[Chapter IV]{dMvS} for this and much more). Now recall that, for a given $C^3$ map $f$, the \textit{Schwarzian derivative} of $f$ is the differential operator defined for all $x$ regular point of $f$ by
 \begin{equation*}
  Sf(x)= \dfrac{D^{3}f(x)}{Df(x)} - \dfrac{3}{2} \left( \dfrac{D^{2}f(x)}{Df(x)}\right)^{2}.
 \end{equation*}

We recall now the definition of an \emph{almost parabolic map}, as given in \cite[Section 4.1, page 354]{dFdM99}.

\begin{definition}\label{def:apm} An \textit{almost parabolic map} is a negative-Schwarzian $C^3$ diffeomorphism$$\phi \colon  J_1\cup J_2\cup \cdots \cup J_\ell \;\to\; J_2\cup J_3\cup \cdots \cup J_{\ell+1},$$such that $\phi(J_k)= J_{k+1}$ for all $1\leq k\leq \ell$, where $J_1,J_2, \ldots, J_{\ell+1}$ are consecutive intervals on the circle (or on the line). The positive integer $\ell$ is called the \textit{length} of $\phi$, and the positive real number
  \[
    \sigma =\min\left\{\frac{|J_1|}{|\cup_{k=1}^\ell J_k|}\,,\, \frac{|J_\ell|}{|\cup_{k=1}^\ell J_k|}     \right\}
  \]is called the \textit{width\/} of $\phi$.
  \end{definition}

The fundamental geometric control on almost parabolic maps is given by the following result.

  \begin{lemma}[Yoccoz's lemma]\label{yoccozlemma}
  Let $\phi \colon \bigcup_{k=1}^\ell J_k \to \bigcup_{k=2}^{\ell+1} J_k$ be an almost parabolic map with length $\ell$ and
  width $\sigma$. There exists a constant $C_\sigma>1$ (depending on $\sigma$ but not on $\ell$) such that,
  for all $k=1,2,\ldots,\ell$, we have
  \begin{equation}\label{yocineq}
    \frac{C_\sigma^{-1}|I|}{[\min\{k,\ell-k\}]^2} \;\leq\; |J_k| \;\leq\;  \frac{C_\sigma|I|}{[\min\{k,\ell-k\}]^2}\ ,
  \end{equation}
  where $I=\bigcup_{k=1}^\ell J_k$ is the domain of $\phi$.
  \end{lemma}

For a proof of Lemma \ref{yoccozlemma} see \cite[Appendix B, page 386]{dFdM99}. To be allowed to use Yoccoz's lemma we will
need the following result.

\begin{lemma}\label{negschwarz} For any given multicritical circle map $f$ there exists $n_0=n_0(f)\in\nt$ such that for any given critical point $c$ of $f$ and for any $n \geq n_0$ we have
that
\[ Sf^{j}(x)<0\quad\text{for all $j\in \{1, \cdots, q_{n+1}\}$ and for all $x \in I_{n}(c)$ regular point of $f^{j}$.}
 \]
Likewise, we have
\[ Sf^{j}(x)<0\quad\text{for all $j\in \{1, \cdots, q_{n}\}$ and for all $x \in I_{n+1}(c)$ regular point of $f^j$}.
 \]
\end{lemma}

For a proof of Lemma \ref{negschwarz} see \cite[Lemma 4.1, page 852]{EdFG18}.

\section{The Katznelson criterion}\label{seckatz}

As stated in the introduction, the proof of Theorem \ref{main} will consist of two separate arguments. The first argument (see \S \ref{secfirststep} below) deals with all irrational rotation numbers {\it except\/} those numbers (of bounded type) whose partial quotients are bounded by a certain constant $B$ that depends only on the real bounds (Theorem \ref{realbounds}). 
The second argument (see \S \ref{secsecstep} below) takes care of the bounded type case. 
The arguments presented in both proofs exploit different aspects of the geometry of multicritical circle maps: the first uses the real bounds and Yoccoz's lemma, whereas the second uses only the real bounds. 

Despite these differences, both parts of the proof will be based on a criterion for non-existence of $\sigma$-finite measures which is a slightly generalized version of a criterion given by Katznelson \cite[Th.~1.1]{katz1}. Consider the following standing hypothesis on the geometry of the dynamical partitions $\mathcal{P}_n(c_0)$ of a $C^1$ minimal homeomorphism $f: S^1\to S^1$ with respect to a given point $c_0\in S^1$. 
\medskip

\noindent{\sl{Standing Hypothesis}\/}. There exist a sequence $\mathbb{N} \ni n_k\to \infty$ of ``good levels'' and constants $1<b_0<b_1$ and $0<\theta<1$ such that the following holds. For each $\Delta\in \mathcal{P}_{n_k}(c_0)$, the collection $\mathcal{A}^{\Delta}=\{J\in \mathcal{P}_{n_k+1}(c_0):\;J\subset \Delta\}$ can be decomposed as a disjoint union $\mathcal{A}^{\Delta}=\mathcal{A}^{\Delta}_1\cup 
\mathcal{A}^{\Delta}_2\cup \mathcal{A}^{\Delta}_3$ with the following properties:
\begin{enumerate}
 \item[(i)] For each $J_1\in \mathcal{A}^{\Delta}_1$ and each $J_2\in \mathcal{A}^{\Delta}_2$ we have $|J_1|\geq b_0|J_2|$;
 \item[(ii)] For each $J_1\in \mathcal{A}^{\Delta}_1$ and each $J_2\in \mathcal{A}^{\Delta}_2$ there exists $k\in \mathbb{N}$ such that $f^k|_{J_1}$ is a diffeomorphism mapping $J_1$ onto $J_2$, and we have $Df^k(x)\geq b_1^{-1}$ for all $x\in J_1$.
 \item[(iii)] We have $\lambda(\Omega)\geq \theta|\Delta|$, where 
 \[
  \Omega = \bigcup_{J\in \mathcal{A}^{\Delta}_1 \cup\mathcal{A}^{\Delta}_2} J\ .
 \]
 \item[(iv)] The sub-collections $\mathcal{A}^{\Delta}_1$ and $\mathcal{A}^{\Delta}_2$ have the same number of elements.{\footnote{Note that nothing is said about the sub-collection $\mathcal{A}^{\Delta}_3$: it plays no role in the arguments to come.}}
 \end{enumerate}
\medskip

\begin{theorem}\label{katzcrit}
 Let $f:S^1\to S^1$ be a $C^1$ minimal homeomorphism satisfying the above standing hypothesis. Then $f$ does not admit a $\sigma$-finite invariant measure which is absolutely continuous with respect to Lebesgue measure. 
\end{theorem}

\begin{proof}
 Assume by contradiction that there exists a $\sigma$-finite measure $\mu$ which is invariant under $f$ and is absolutely continuous with respect to Lebesgue measure. 
 Let $\psi=d\mu/d\lambda$ be the corresponding Radon-Nikodym derivative. This is a Borel measurable function which is positive and finite Lebesgue a.e., and it satisfies the cocycle identity \eqref{radon}. By an easy induction, that cocycle identity can be written more generally as
 \begin{equation}\label{kradon}
  \psi(x) = 
  \psi\circ f^k(x)\cdot Df^k(x)\ \ \ \textrm{for Lebesgue a.e.}\ x\in S^1\ , \ \textrm{for all}\ k\in \mathbb{Z}\ .
  \end{equation}  
 Fix a small number $0<\delta<1$; we will need it small enough that $(1+\delta)^{-1}b_0>1$. For each real number $c$ consider the Borel set $E_c=\{x\in S^1:\,c\leq \psi(x)\leq c(1+\delta)\}$. Then we have $\lambda(E_c)>0$ for some choice of $c$. We choose such $c$ and from now on write $E=E_c$. 
 
 By the Lebesgue density theorem, $\lambda$-a.e. $x\in E$ is such that the density of $E$ at $x$ is $1$. Hence for each $\epsilon>0$ we can find a good level $n_k\in \mathbb{N}$ and an atom $\Delta\in \mathcal{P}_{n_k}(c_0)$ such that 
 \begin{equation}\label{nosigma1}
  \frac{\lambda(E\cap \Delta)}{|\Delta|} \geq 1-\epsilon\ .
 \end{equation}
 We will show that the assumption at the start of this proof contradicts our standing hypothesis on $f$ if we take $\epsilon$ sufficiently small. How small $\epsilon$ has to be will be determined in the course of the argument to follow. 
 
 Let $\mathcal{A}^\Delta$ and $\mathcal{A}^\Delta_i$, $i=1,2,3$ be as defined before, and for $i=1,2$ let $\Omega_i=\bigcup_{J\in \mathcal{A}^\Delta_i} J$.
 Then (iii) in our standing hypothesis tells us that $\Omega=\Omega_1\cup \Omega_2$ satisfies $\lambda(\Omega)\geq \theta|\Delta|$. Hence from \eqref{nosigma1} we have
 \begin{equation}\label{nosigma2}
  \frac{\lambda(E\cap \Omega)}{\lambda(\Omega)} \geq 1-\epsilon\theta^{-1}\ ,
 \end{equation}
provided $\epsilon$ is so small that $\epsilon\theta^{-1}<1$. Note that our standing hypothesis also tells us that $b_0\lambda(\Omega_2) \leq \lambda(\Omega_1) \leq b_1\lambda(\Omega_2)$. 
These inequalities imply that
\begin{equation}\label{nosigma3}
 \lambda(\Omega)\leq (1+b_0^{-1})\lambda(\Omega_1)\ \ \textrm{and}\ \ 
 \lambda(\Omega)\leq (1+b_1)\lambda(\Omega_2) \ .
\end{equation}
Using \eqref{nosigma2} and the first inequality in \eqref{nosigma3}, we get
\begin{align*}
 \lambda(\Omega_1) &\leq \lambda(E\cap \Omega_1) + \lambda(\Omega\setminus E) \\
 {} &\leq \lambda(E\cap \Omega_1) + \epsilon\theta^{-1}\lambda(\Omega) \\
 {} &\leq \lambda(E\cap \Omega_1) + \epsilon\theta^{-1}(1+ b_0^{-1})\lambda(\Omega_1)\ .
\end{align*}
Hence we have 
\begin{equation}\label{nosigma4}
  \frac{\lambda(E\cap \Omega_1)}{\lambda(\Omega_1)} \geq 1-\epsilon\theta^{-1}(1+b_0^{-1})\ ,
 \end{equation}
and this lower bound will be positive (in fact close to one) provided $\epsilon$ is sufficiently small. 
Similarly, using \eqref{nosigma2} and the second inequality in \eqref{nosigma3}, we deduce that
\begin{equation}\label{nosigma5}
  \frac{\lambda(E\cap \Omega_2)}{\lambda(\Omega_2)} \geq 1-\epsilon\theta^{-1}(1+b_1)\ .
 \end{equation}
Thus, writing $\eta=\epsilon\theta^{-1}\max\{1+b_0^{-1}\,,\,1+b_1\}=\epsilon\theta^{-1}(1+b_1)$, we have
\begin{equation}\label{nosigma6}
  \frac{\lambda(E\cap \Omega_i)}{\lambda(\Omega_i)} \geq 1-\eta\ ,\ \ 
  \textrm{for}\ \ i=1,2\ .
\end{equation}
Note that $\eta\to 0$ when $\epsilon\to 0$. Now, since both $\Omega_1$ and $\Omega_2$ are disjoint unions of atoms in $\mathcal{P}_{n_k+1}(c_0)$, it follows from \eqref{nosigma6} that there exist atoms $J_1\in \mathcal{A}^\Delta_1$ and $J_2\in \mathcal{A}^\Delta_2$ such that 
\begin{equation}\label{nosigma7} 
\lambda(J_i\cap E)\geq (1-\eta)|J_i|\ ,\ \  \textrm{for}\  i=1,2\ .
\end{equation}
Let $k\in \mathbb{N}$ be such that $f^k$ maps $J_1$ diffeomorphically onto $J_2$, and let us estimate the Lebesgue measure of $f^{-k}(J_2\setminus E)$.
By (ii) in our standing hypothesis and the chain rule we have $Df^{-k}(y)\leq b_1$ for all $y\in J_2$. Since by \eqref{nosigma7} we have 
$\lambda(J_2\setminus E)\leq \eta|J_2|$, we get
\begin{equation}\label{nosigma8}
 \lambda(f^{-k}(J_2\setminus E)) = \int_{J_2\setminus E} Df^{-k}\,d\lambda \leq b_1\eta|J_2|\ .
\end{equation}
Letting $J_1^*=\{x\in J_1\cap E:\, f^k(x)\in E\}$, it follows from \eqref{nosigma7} and \eqref{nosigma8} that 
\begin{equation}\label{nosigma9}
 \lambda(J_1^*) =\lambda(J_1\cap E) - \lambda(f^{-k}(J_2\setminus E))
 \geq [(1-\eta)b_0 -\eta b_1]\,|J_2|\ .
\end{equation}
But now observe that the equality $\psi=(\psi\circ f^k)Df^k$ holds Lebesgue almost everywhere: this is simply the cocycle identity \eqref{kradon}. Since for every $x\in J_1^*$ we have both $x\in E$ and $f^k(x)\in E$, it follows from this equality and the definition of $E$ that for Lebesgue a.e. $x\in J_1^*$ we have 
$Df^k(x)\geq (1+ \delta)^{-1}$. Therefore 
\begin{equation}\label{nosigma10}
 |J_2| > \lambda(f^k(J_1^*)) = \int_{J_1^*} Df^k\,d\lambda \geq (1+\delta)^{-1}\lambda(J_1^*)\ .
\end{equation}
Combining \eqref{nosigma9} and \eqref{nosigma10} and cancelling out $|J_2|$ from both sides of the resulting inequality, we deduce at last that
\begin{equation}\label{nosigma11}
 (1+ \delta)^{-1}[(1-\eta)b_0 - \eta b_1] < 1\ .
\end{equation}
But since $(1+\delta)^{-1}b_0>1$, the inequality \eqref{nosigma11} is clearly violated if $\eta$ is sufficiently small, which is certainly the case if we choose $\epsilon$ sufficiently small. We have reached the desired contradiction, and the proof is complete. 
\end{proof}

\begin{remark}\label{shortcrit}
 A close inspection of the proof shows that we do not need the full strength of the standing hypothesis. All we need is that, given any interval $I$ on the circle, we can find inside it two disjoint intervals $J',J''$, both comparable in size with $I$, with $|J'|$ greater than $|J''|$ by a definite factor, and an iterate of $f$ mapping $J'$ onto $J''$ with bounded distortion. 
\end{remark}

\section{Proof of Theorem \ref{main}}\label{secproofmain}

We are now ready for the two major steps in the proof of Theorem \ref{main}. 

\subsection{First step}\label{secfirststep} The precise result we shall prove here is the following  weaker version of Theorem \ref{main}. 

\begin{theorem}\label{nosigcrit}
 Given $N\geq 1$ in $\nt$ and $d>1$ there exists a universal constant $B=B(N,d)\in \nt$ such that the following holds. If $f$ is a multicritical circle map with at most $N$ critical points whose criticalities are bounded by $d$, and if the rotation number of $f$ is irrational and its
 partial quotients $a_n$ satisfy $\limsup a_n \geq B$, then $f$ does not admit an invariant $\sigma$-finite measure which is absolutely continuous with respect to Lebesgue measure.
\end{theorem}

In the proof of Theorem \ref{nosigcrit}, we will make extensive use of the following fact, which is an immediate consequence of \cite[Lemma 4.2, page 5600]{EdF18}.


\begin{lemma}\label{lemmaiteratescritspots} Let $c_0$ be a critical point of $f$, and let $0\leq k< a_{n+1}$ be such that the interval $f^{q_n+kq_{n+1}}\big(I_{n+1}(c_0)\big) \subset I_n(c_0)$ contains a critical point of $f^{q_{n+1}}$. Then$$\left|f^i\big(f^{q_n+kq_{n+1}}(I_{n+1}(c_0))\big)\right|\asymp \left|f^i\big(I_n(c_0)\big)\right|\quad\mbox{for all $i\in\{0,1,...,q_{n+1}\}$.}$$
\end{lemma}

\begin{proof}
We only sketch the proof. 
For $i=0$ the statement is just \cite[Lemma 4.2, page 5600]{EdF18}. Moreover, by Theorem \ref{realbounds}, the image of each critical spot under $f^{q_{n+1}}$ is also comparable to $I_n(c_0)$: this is simply because $f^{q_{n+1}}\big(f^{q_n+kq_{n+1}}(I_{n+1}(c_0))\big)=f^{q_n+(k+1)q_{n+1}}\big(I_{n+1}(c_0)\big)$ is adjacent to $f^{q_n+kq_{n+1}}\big(I_{n+1}(c_0)\big)$ in $\mathcal{P}_{n+1}(c_0)$. So the statement of our lemma also holds for $i=q_{n+1}$. Now, for each $i\in\{1,...,q_{n+1}-1\}$ consider the iterate $f^{q_{n+1}-i}$, and apply the Cross-Ratio Inequality from Section \ref{secprel}. For more details, see \cite[Section 4.4, page 5602]{EdF18}. 
\end{proof}

Following the terminology of \cite{EdF18}, an interval such as $f^{q_n+kq_{n+1}}\big(I_{n+1}(c_0)\big)$ appearing in the statement above, containing some critical point of $f^{q_{n+1}}$, is called a {\it critical spot\/}. Thus, Lemma \ref{lemmaiteratescritspots} is saying that every critical spot is large, {\it i.e.,} is comparable to the atom of $\mathcal{P}_n(c_0)$ in which it is contained, and the same happens to all its images up to time $i=q_{n+1}$.


\begin{proof}[Proof of Theorem \ref{nosigcrit}] By Theorem \ref{katzcrit}, it suffices to show that an $f$ as in the statement satisfies the {\it standing hypothesis\/} previously formulated, provided $\limsup{a_n}$ is sufficiently large. This will be proved with the help of the real bounds (Theorem \ref{realbounds}), Yoccoz's inequality (Lemma \ref{yoccozlemma}) and Lemma \ref{lemmaiteratescritspots} above.
 
Let $c_0$ be a critical point of $f$ and consider the associated dynamical partitions $\mathcal{P}_n(c_0)$ for $n\geq n_0(f)$, where $n_0(f)$ is as in Theorem \ref{realbounds}. We are also assuming that such $n$ is large enough that the iterates $f^{q_n}$ and $f^{q_{n+1}}$ have negative Schwarzian derivative at all points in $I_{n+1}(c_0)$ ($I_{n}(c_0)$ respectively) where their derivatives do not vanish (this is possible by Lemma \ref{negschwarz}). We will only consider in the proof long atoms of $\mathcal{P}_n(c_0)$, the proof for the short ones being the same. Moreover, we will decompose first the collection $\big\{J\in \mathcal{P}_{n+1}(c_0):\;J\subset I_n(c_0)\big\}$, and then we will spread this decomposition iterating by $f$. So let $\Delta=I_n(c_0)$, and consider the following consecutive atoms of $\mathcal{P}_{n+1}(c_0)$ inside $\Delta$: $\Delta_0=f^{q_n}(I_{n+1})$ and $\Delta_j=f^{jq_{n+1}}(\Delta_0)$ for $j=1,2,\ldots, a_{n+1}-1$; note that $\Delta_j=f^{q_{n+1}}(\Delta_{j-1})$ for all $1\leq j\leq a_{n+1}-1$. Some of these intervals may be critical spots (which are always comparable in size with $|\Delta|$, by Lemma \ref{lemmaiteratescritspots}). We look at the {\it bridges\/} between such critical spots, and pick the longest one. More precisely, let $0\leq j_1 \leq j_2\leq a_{n+1}-1$ with $j_2-j_1$ maximal with the property that $\phi=f^{q_{n+1}}|_{\Delta_{j_1}\cup\cdots\cup\Delta_{j_2}}$ is a diffeomorphism onto its image. Let $T_n=\Delta_{j_1}\cup\cdots\cup\Delta_{j_2}$, $R_n=\Delta_{j_1}$, $L_n=\Delta_{j_2}$ and $M_n=T_n\setminus(L_n \cup R_n)=\Delta_{j_1+1}\cup\cdots\cup\Delta_{j_2-1}$. Note that $\phi|_{M_n}$ is an {\it almost parabolic map\/} (see Definition \ref{def:apm}) with length $\ell=j_2-j_1-1$, and note that $\ell \geq a_{n+1}/(N+1)$, where $N$ is the number of critical points of $f$. Let us write $J_1=\Delta_{j_1+1}\,,\,J_2=\Delta_{j_1+2}\,,\,\ldots\,,\,J_{\ell}=\Delta_{j_1+\ell}=\Delta_{j_2-1}\/$. From the real bounds (Theorem \ref{realbounds}), we have $|J_1|\asymp |\Delta|\asymp |J_\ell|$, with {\it beau\/} comparability constants. Therefore, by Yoccoz's inequality (Lemma \ref{yoccozlemma}), there exists a constant $C_0>1$, depending only on $f$, such that, for all $1\leq j\leq \ell$,  
 \begin{equation}\label{nosigcrit1}
  \frac{C_0^{-1}}{\min\{j\,,\,\ell-j\}^2} \;\leq\; \frac{|J_j|}{|\Delta|}
  \leq \frac{C_0}{\min\{j\,,\,\ell-j\}^2} 
\end{equation}
Now we claim that there exists a constant $\tau>0$ (depending only on $f$) such that$$\big|f^i(L_n)\big|>\tau\,\big|f^i(M_n)\big|\quad\mbox{and}\quad\big|f^i(R_n)\big|>\tau\,\big|f^i(M_n)\big|$$for all $i\in\{ 0, \cdots, q_{n+1}\}$. Indeed, again by combining Theorem \ref{realbounds} with Lemma \ref{lemmaiteratescritspots} we obtain the claim for both $i=0$ and $i=q_{n+1}$. By the Cross-Ratio Inequality (note that the intervals $T_n,f(T_n),...,f^{q_{n+1}-1}(T_n)$ are pairwise disjoint), we deduce the claim for any $i\in\{1, \cdots, q_{n+1}-1\}$. With this at hand, and since $f^i|_{T_n}$ is a diffeomorphism for any $i\in\{0,...,q_{n+1}\}$, we can apply Koebe distortion principle (Lemma \ref{koebe}) in order to obtain a constant $K=K(f)>1$ such that $f^i|_{M_n}$ has distortion bounded by $K$ for each $i\in\{0, \cdots, q_{n+1}\}$. Let us now define $B=2(N+1)\lceil\sqrt{2K}C_0\rceil+1$. We are assuming from now on that $n$ is one of infinitely many natural numbers such that $a_{n+1}\geq B$. Let $m$ be the smallest natural number such that $KC_0^2m^{-2}\leq \frac{1}{2}$; in other words, let $m=\lceil\sqrt{2K}C_0\rceil$. Since $a_{n+1}\geq B$, we have 
\[
 \frac{\ell}{2}\;\geq\;\frac{a_{n+1}}{2(N+1)}\;\geq\;\frac{B}{2(N+1)}\;>\;\lceil\sqrt{2K}C_0\rceil = m\ .
\]
Thus, setting $J'=J_1$ and $J''=\phi^{m-1}(J')=J_m$, it follows from \eqref{nosigcrit1} that 
\begin{equation}\label{nosigcrit2}
 \frac{1}{C_0^2m^2} \;\leq\; \frac{|J''|}{|J'|}\;\leq\; \frac{C_0^{2}}{m^2}\;\leq\; \frac{1}{2K} \;<\; \frac{1}{2} \ .
\end{equation}
We are now ready to define the desired decomposition of $\mathcal{A}^{\Delta}$, the collection of all atoms of $\mathcal{P}_{n+1}(c_0)$ that are contained in $\Delta=I_n(c_0)$. Let $\mathcal{A}^{\Delta}_1 =\{J'\}$, let $\mathcal{A}^{\Delta}_2=\{J''\}$ and let 
$\mathcal{A}^{\Delta}_3=\mathcal{A}^{\Delta}\setminus (\mathcal{A}^{\Delta}_1 \cup 
\mathcal{A}^{\Delta}_2)$. We claim that this decomposition satisfies all conditions (i)-(iv) in the standing hypothesis. From \eqref{nosigcrit2}, we have $|J'|\geq 2|J''|$, so (i) is satisfied with $b_0=2$. By the mean value theorem, there exists $\xi\in J'$ such that 
\[
 D\phi^{m-1}(\xi) \;=\; \frac{|J''|}{|J'|}\;\geq\; \frac{1}{C_0^2m^2}\ ,
\]
where we have again used \eqref{nosigcrit2}. By Koebe distortion principle, there exists $C_1>1$ (depending only on $f$) such that 
\[
 C_1^{-1}\;\leq\; \frac{D\phi^{m-1}(x)}{D\phi^{m-1}(\xi)}\;\leq\; C_1\ ,\ \ \ \textrm{for all}\ x\in J'\ . 
\]
Combining these facts we deduce that $D\phi^{m-1}(x)\geq (C_0^2C_1 m^2)^{-1}$, and so (ii) is certainly satisfied if we take $k=q_{n+1}(m-1)$ and $b_1=KC_0^2C_1m^2=KC_0^2C_1\lceil\sqrt{2K}C_0\rceil^2$. Note that 
$b_1>2=b_0$. For $\Omega=J'\cup J''$, we now have, using \eqref{nosigcrit1}, the simple bound 
$
\lambda(\Omega)=|J'|+|J''|\geq |J'|\geq C_0^{-1}|\Delta| \ .
$
This shows that (iii) is satisfied if we choose $\theta=C_0^{-1}<1$.  
Finally, condition (iv) is trivially satisfied because both $\mathcal{A}^{\Delta}_1$ and $\mathcal{A}^{\Delta}_2$ have a single element. 

Now we spread the previous decomposition along the whole family of long intervals of $\mathcal{P}_n(c_0)$. More precisely, for each $i\in\{1,...,q_{n+1}-1\}$ we define a decomposition of $\mathcal{A}^{\Delta}$, the collection of all atoms of $\mathcal{P}_{n+1}(c_0)$ that are contained in $\Delta=f^i\big(I_n(c_0)\big)$, as follows: let $\mathcal{A}^{\Delta}_1 =\{f^i(J')\}$, let $\mathcal{A}^{\Delta}_2=\{f^i(J'')\}$ and let 
$\mathcal{A}^{\Delta}_3=\mathcal{A}^{\Delta}\setminus (\mathcal{A}^{\Delta}_1 \cup 
\mathcal{A}^{\Delta}_2)$. Again, we claim that this decomposition satisfies all conditions (i)-(iv) in the standing hypothesis. Indeed, for each $i\in\{1,...,q_{n+1}-1\}$ let $x_i' \in J'$ and $x_i'' \in J''$ be given by the mean value theorem:$$\frac{\big|f^{i}(J'')\big|}{\big|f^{i}(J')\big|}=\frac{Df^i(x_i'')}{Df^i(x_i')}\,\frac{|J''|}{|J'|}\,.$$By bounded distortion and \eqref{nosigcrit2} we obtain$$\frac{\big|f^{i}(J'')\big|}{\big|f^{i}(J')\big|}=\frac{Df^i(x_i'')}{Df^i(x_i')}\,\frac{|J''|}{|J'|} \leq K\,\frac{|J''|}{|J'|} \leq \frac{K\,C_0^2}{m^2} \leq \frac{1}{2}\,.$$So (i) is again satisfied with $b_0=2$. Now if we conjugate $\phi^{m-1}:J' \to J''$ with the iterate $f^i$, we obtain a diffeomorphism $f^i \circ \phi^{m-1} \circ f^{-i}:f^{i}(J') \to f^{i}(J'')$ which satisfies the following for all $x \in f^i(J')$:
\begin{align*}
D\big(f^i \circ \phi^{m-1} \circ f^{-i}\big)(x)&=D\phi^{m-1}\big(f^{-i}(x)\big)\,Df^i\big(\phi^{m-1} \circ f^{-i}(x)\big)\,Df^{-i}(x)=\\
&=D\phi^{m-1}\big(f^{-i}(x)\big)\,\frac{Df^i\big(\phi^{m-1} \circ f^{-i}(x)\big)}{Df^i\big(f^{-i}(x)\big)}\,.
\end{align*}
Since $f^{-i}(x)$ belongs to $J'$, $\phi^{m-1}\big(f^{-i}(x)\big)$ belongs to $J''$ and then$$D\big(f^i \circ \phi^{m-1} \circ f^{-i}\big)(x)\geq\frac{1}{K}\,D\phi^{m-1}\big(f^{-i}(x)\big)\geq\frac{1}{K}\,(C_0^2C_1 m^2)^{-1}\,.$$Therefore, just as before, (ii) is again satisfied with $k=q_{n+1}(m-1)$ and $b_1=KC_0^2C_1m^2=KC_0^2C_1\lceil\sqrt{2K}C_0\rceil^2$. By Lemma \ref{lemmaiteratescritspots}, the $i$-th iterate of a critical spot, contained in $I_n(c_0)$, is comparable to $f^{i}\big(I_{n}(c_0)\big)$ for all $i\in\{0,1,...,q_{n+1}\}$ and then, by Theorem \ref{realbounds}, the interval $f^{i}(J')$ is comparable to $f^{i}\big(I_n(c_0)\big)$ as well, which implies (iii). Again, condition (iv) is trivially satisfied. Summarizing, we have shown that, for infinitely many values of $n$, the partitions $\mathcal{P}_n(c_0)$ satisfy conditions (i) through (iv) of the standing hypothesis. Therefore, by Theorem \ref{katzcrit}, $f$ does not admit a $\sigma$-finite invariant measure equivalent to Lebesgue measure. This finishes the proof. 
\end{proof}

\subsection{Second step}\label{secsecstep} We now move to the bounded type case. Here our goal will be to prove the following result.

\begin{theorem}\label{nosigbound} If $f$ is a multicritical circle map with an irrational rotation number of bounded type, then $f$ does not admit an invariant $\sigma$-finite measure which is absolutely continuous with respect to Lebesgue measure.
\end{theorem}

In the proof of Theorem \ref{nosigbound} we will make use of the following two auxiliary results.

\begin{prop}\label{expdecay} Given a multicritical circle map $f$ with an irrational rotation number of bounded type, there exist constants $C_0>1$ and  $0<\lambda_0<\lambda_1<1$ with the following property. For each $x\in S^1$, each $n,k\geq 0$ and every pair of atoms $I\in \mathcal{P}_n(x)$ and $J\in \mathcal{P}_{n+k}(x)$ with $J\subseteq I$, we have
 \[
  C_0^{-1}\lambda_0^k\;\leq\; \frac{|J|}{|I|} \;\leq\; C_0\lambda_1^k \ .
 \]
\end{prop}

\begin{prop}\label{propnegschw} Given a multicritical circle map $f$ with an irrational rotation number of bounded type, there exists $n_0=n_0(f)\in\nt$ such that for all $n \geq n_0$ we have$$Sf^{q_{n+1}}(x)<0\quad\text{for all $x \in S^1$ regular point of $f^{q_{n+1}}$.}$$Likewise, we have$$Sf^{q_{n}}(x)<0\quad\text{for all $x \in S^1$ regular point of $f^{q_{n}}$}.$$
\end{prop}

We postpone the proof of both Proposition \ref{expdecay} and Proposition \ref{propnegschw} until Appendix \ref{app}. We emphasize that the statement of Proposition \ref{expdecay} is {\it false\/} for unbounded combinatorics. On the other hand, Proposition \ref{propnegschw} is most likely true for \emph{any} irrational rotation number (however, this more general fact will not be needed in this paper).

Our proof of Theorem \ref{nosigbound} will be based on the following lemma. Recall that we are fixing our attention on a critical point $c$ of $f$. Below, we use the following notation: for all $i\geq 0$, let  $c_{-i}=f^{-i}(c)$; we write accordingly  $I_n(c_{-i})=f^{-i}(I_n(c))$ for all $n\geq 0$ and all $i\geq 0$. 

\begin{lemma}\label{cruciallemma}
 There exist constants $K>1$ and  $0<\theta<1$ such that the following holds for all $n$ sufficiently large and each $0\leq i<q_{n}$. There exist subintervals $\Delta_{i,n}'\subset I_{n+1}(c_{-i})$ and $\Delta_{i,n}''\subset I_n(c_ {-i})$  such that
 \begin{enumerate}
  \item[(i)] $\Delta_{i,n}'\cap \Delta_{i,n}'' = \textrm{\O}$;
  \item[(ii)] $|\Delta_{i,n}'|\geq 2|\Delta_{i,n}''|$; 
  \item[(iii)] $|\Delta_{i,n}''|\geq \theta |I_n(c_{-i})|$;
  \item[(iv)] $\Delta_{i,n}''=f^{q_n}(\Delta_{i,n}')$, and $f^{q_n}|_{\Delta_{i,n}'}:\,\Delta_{i,n}'\to \Delta_{i,n}''$ is a diffeomorphism whose distortion is bounded by $K$.
 \end{enumerate}
\end{lemma}

\begin{proof} We assume from the start that $n$ is so large that $f^{q_n}|_{I_{n+1}(c_{-i})}$ has negative Schwarzian derivative for all $0\leq i<q_{n}$. This is possible by Proposition \ref{propnegschw}. {\emph{Note that each $c_{-i}$ for $0\leq i<q_{n+1}$ is a critical point of $f^{q_n}$}}. In what follows, we keep $n$ and $0\leq i<q_{n}$ fixed. 

Note that for all $k\geq 0$ {\it even} we have $I_{n+k+1}(c_{-i})\subseteq I_{n+1}(c_{-i})$. By Proposition \ref{expdecay}, there exist constants $0<\lambda_0<\lambda_1<1$ and $C_0>1$ such that
 \begin{equation}\label{largeandsmall1}
  C_0^{-1}\lambda_0^k\;\leq\; \frac{|I_{n+k+1}(c_{-i})|}{|I_{n}(c_{-i})|} \;\leq\; C_0\lambda_1^k \ .
 \end{equation}
Moreover, if we denote by $d=d(i,n)>1$ the power-law at the critical point $c_{-i}$ of $f^{q_n}$, then we have{\footnote{One can easily check that $d_{\min}\leq d(i,n)\leq d_{\max}^N$, where  $d_{\min}$ and $d_{\max}$ are the smallest and largest power-law exponents of the critical points of $f$, and $N$ is the number of such critical points.}} 
\begin{equation}\label{largeandsmall2}
 \frac{|f^{q_n}(I_{n+k+1}(c_{-i}))|}{|I_n(c_{-i})|} \;\asymp\; 
 \left(  \frac{|I_{n+k+1}(c_{-i})|}{|I_n(c_{-i})|}    \right)^d\ \ .
\end{equation}
Let us write $I=I_{n+k+1}(c_{-i})$ and $J=f^{q_n}(I)$; these are obviously disjoint intervals (see  Figure \ref{fig1}), and they are both atoms of $\mathcal{P}_{n+k}(c_{-i})$. 
Combining \eqref{largeandsmall1} with \eqref{largeandsmall2}, we deduce that there exists a constant $C_1>1$ (independent of $n$ and $k$) such that 
\begin{equation}\label{largeandsmall3}
 C_1^{-1} \lambda_0^{k(d-1)}|I|\;\leq\; |J| \;\leq\;
 C_1 \lambda_1^{k(d-1)}|I|
\end{equation}
Note that $f^{q_n}|_{I}:\,I\to J$ has at most $N$ critical points{\footnote{Again, $N$ is the total number of critical points of $f$.}}, and has negative Schwarzian at all regular points.
Note that, by choosing $k$ sufficiently large, we can make $|J|$  definitely smaller than $|I|$. The meaning of ``definitely smaller'', and thus how large $k$ has to be, will be clear in a moment.

\begin{figure}[t]
\begin{center}~
\hbox to \hsize{\psfrag{0}[][][1]{$0$} \psfrag{I}[][][1]{$I_n(c_{-i})$}
\psfrag{I1}[][][1]{$I_{n+1}(c_{-i})$}
\psfrag{f}[][][1]{$f^{q_n}$}
\psfrag{K}[][][1]{$I_{n+k+1}(c_{-i})$}
\psfrag{T}[][][1]{$\!T$}
\psfrag{d}[][][1]{$\Delta_{i,n}'$}
\psfrag{D}[][][1]{$\Delta_{i,n}''$}
\psfrag{c}[][][1]{$\;\;c_{-i}$}
\hspace{1.0em} \includegraphics[width=6.0in]{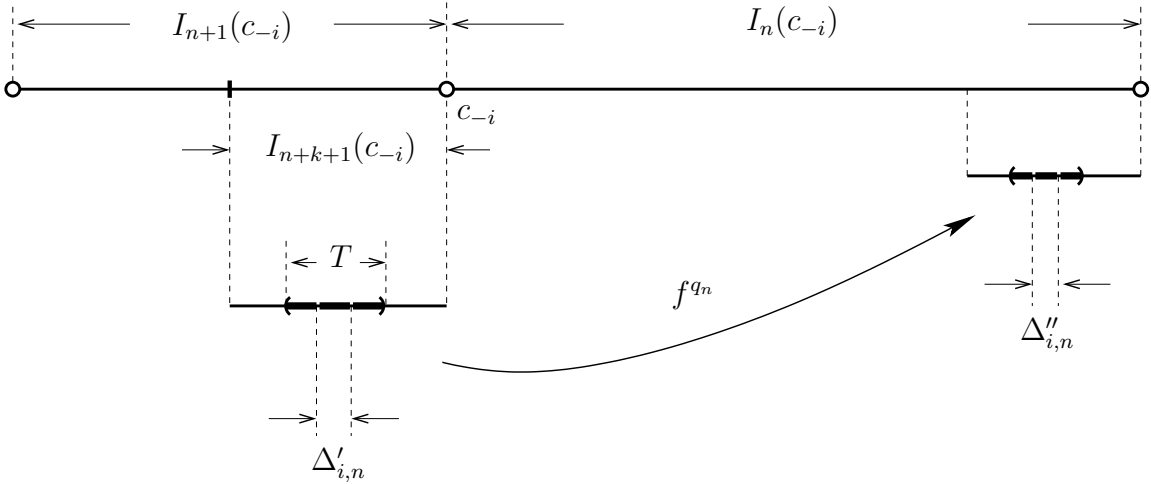}
   }
\end{center}
\caption[fig1]{\label{fig1} The iterate $f^{q_n}$ maps $\Delta_{i,n}'$ diffeomorphically onto $\Delta_{i,n}''$ with bounded distortion. }
\end{figure}

For $p\geq 0$, let us denote the number of atoms of $\mathcal{P}_{n+k+p}(c_{-i})$ inside $I$ (or $J$) by $a=a(n,k,p)$. Then we have $2^{p}\leq a\leq (A+1)^p$ (where $A=\sup{a_n}<\infty$ is the least upper bound on the convergents of the rotation number of $f$). Choose $p=p(N)$ smallest with the property that $2^{p}>3N+2$. Since $f^{q_n}|_{I}$ has at most $N$ critical points, and since $a>3N+2$, it follows from the pigeonhole principle that there exist $3$ consecutive atoms of $\mathcal{P}_{n+k+p}(c_{-i})$ inside $I$, say $L, M, R$, such that the open interval $T=\textrm{int}(L\cup M\cup R)$ contains no critical point of $f^{q_n}$. Hence $f^{q_n}|_{T}:T\to f^{q_n}(T)$ is a diffeomorphism with negative Schwarzian derivative. Applying Koebe's non-linearity principle, we see that
\begin{equation}\label{nonlin}
 |D\log{Df^{q_n}(x)}|\;\leq\; \frac{2}{\tau}\ \ \ \textrm{for all}\ x\in M\ .
\end{equation}
where $\tau$ is the {\it space\/} of $M$ inside $T$, namely
\[
 \tau\;=\; \min\left\{\frac{|L|}{|M|}\,,\,\frac{|R|}{|M|}\right\}\ .
\]
From the real bounds, we know that $\tau\geq C_2$, for some constant $C_2>0$. Using this fact in  \eqref{nonlin} and integrating the resulting inequality, we deduce that 
\begin{equation}\label{bounddist}
 e^{-2/C_2}\;\leq\; \frac{Df^{q_n}(x)}{Df^{q_n}(y)}\;\leq\; e^{2/C_2}\ ,\ \ \ \ \textrm{for all}\ x,y\in M\ .
\end{equation}
Now, applying once again Proposition \ref{expdecay} (note that we are using the bounded type hypothesis!), it follows that there exists a constant $C_3>1$ depending on $A$ such that 
\begin{equation}\label{eqM1}
 C_3^{-1}\lambda_0^p\;\leq\; \frac{|M|}{|I|}\;\leq\; C_3\lambda_1^p\ ,
\end{equation}
as well as 
\begin{equation}\label{eqM2}
 C_3^{-1}\lambda_0^p\;\leq\; \frac{|f^{q_n}(M)|}{|J|}\;\leq\; C_3\lambda_1^p\ ,
\end{equation}
Putting together \eqref{largeandsmall3}, \eqref{eqM1} and \eqref{eqM2}, we deduce that 
\begin{equation}\label{eqM3}
 |M|\;\geq\; C_1^{-1} C_3^{-2}\lambda_0^p  \lambda_1^{-k(d-1)-p}
 \,|f^{q_n}(M)| \ .
\end{equation}
Likewise, putting together \eqref{largeandsmall1}, \eqref{largeandsmall3} and \eqref{eqM3}, we get
\begin{equation}\label{eqM3.5}
 |f^{q_n}(M)|\;\geq\; (C_0C_1C_3)^{-1}\lambda_0^{kd+p}|I_n(c_{-i})|\ .
\end{equation}
Now let us choose $k\geq 1$ smallest with the property that 
\begin{equation}\label{eqM4}
C_1^{-1} C_3^{-2}\lambda_0^p  \lambda_1^{-k(d_0-1)-p}\;\geq\;2\ ,
\end{equation}
where $d_0=\min_{i,n}d(i,n)>1$
Such $k$ exists (and is independent of $n$) because $\lambda_1<1$. 

To finish the proof, we define $\Delta_{i,n}'=M$ and $\Delta_{i,n}''=f^{q_n}(M)$. These, we claim, are the intervals satisfying properties (i)-(iv) in the statement. Indeed, property (i) is clear. Property (iv) follows directly from \eqref{bounddist} if we take $K=e^{2/C_2}$. Property (ii) follows from inequalities \eqref{eqM3} and \eqref{eqM4}. 
Finally, property (iii) follows from \eqref{eqM3.5}, provided we take $\theta=(C_0C_1C_3)^{-1}\lambda_0^{kd+p}$. 
The proof is complete.
\end{proof}

\begin{proof}[Proof of Theorem \ref{nosigbound}] 
The proof will based on the generalized Katznelson criterion given by Theorem \ref{katzcrit}. Our argument combines Lemma \ref{cruciallemma} with the Cross Ratio Inequality. 

It is enough to show that $f$ satisfies the {\sl standing hypothesis\/} stated prior to Theorem~\ref{katzcrit} concerning the sequence of dynamical partitions $\mathcal{P}_n(c)$ for some choice of critical point $c$. For this purpose, as we have seen in the proof of that theorem (see also Remark \ref{shortcrit}), it suffices to prove the following statement.
\medskip 

\noindent{\it Claim.\/ For every sufficiently large $n$, every atom 
$I\in \mathcal{P}_n(c)$ contains two disjoint subintervals $\Delta',\Delta''$ such that: (a) $|\Delta'|\geq 2|\Delta''|$; (b) $|\Delta'|\asymp |I|\asymp |\Delta''|$; (c) there exists $q\geq 1$ such that $\Delta''=f^{q}(\Delta')$ and $f^{q}|_{\Delta'}: \Delta'\to \Delta''$ is a diffeomorphism with bounded distortion.}{\footnote{The claim's proof will show that $q=q_n$ or $q=q_{n+1}$, depending on whether $I$ is a {\it long} or {\it short} atom of $\mathcal{P}_n(c)$, respectively.}} 
\medskip

\begin{figure}[t]
\begin{center}~
\hbox to \hsize{
\psfrag{c}[][][1]{$c$} 
\psfrag{ci}[][][1]{$c_{-i}$} 
\psfrag{I}[][][1]{$I=f^{q_{n+1}-i}(I_n(c))$}
\psfrag{I0}[][][1]{$I_{n+1}(c)$}
\psfrag{I1}[][][1]{$I_{n}(c)$}
\psfrag{fq}[][][1]{$\;f^{q_{n+1}-i}$}
\psfrag{fi}[][][1]{$f^i$}
\psfrag{J}[][][1]{$J_{n+4}$}
\psfrag{Ji}[][][1]{$\!\!f^{-i}(J_{n+4})$}
\psfrag{T}[][][1]{$f^{q_{n+1}}(I_n(c))$}
\psfrag{d}[][][1]{$\Delta''$}
\psfrag{D}[][][1]{$\Delta'$}
\hspace{1.0em} \includegraphics[width=5.9in]{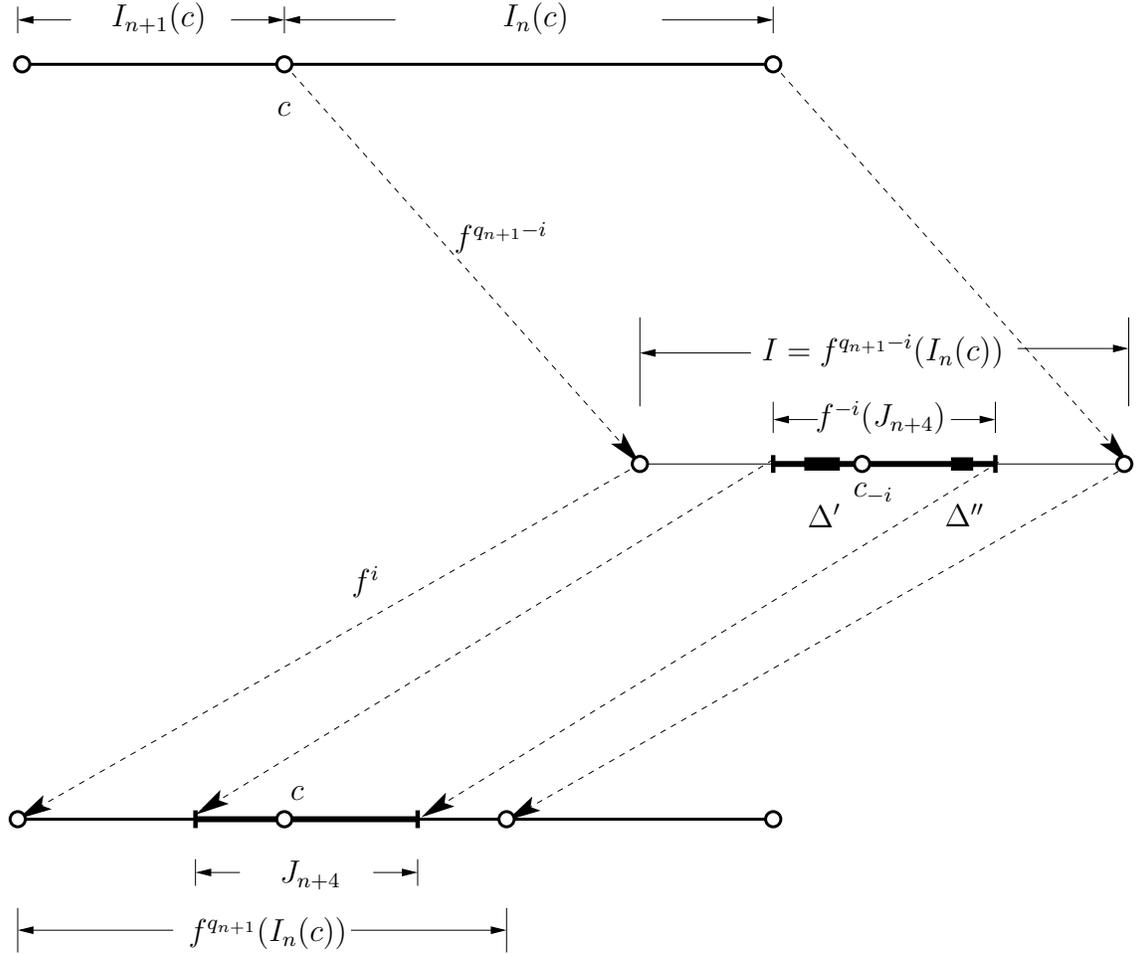}
   }
\end{center}
\caption[fig2]{\label{fig2} Finding two intervals, long and short, inside an atom $I\in \mathcal{P}_n(c)$. }
\end{figure}

The comparability constants and bounds implicit in this statement depend only on the real bounds for $f$ and the bound on the combinatorics.  
 To simplify the notation a bit, let us write $J_k=I_k(c)\cup I_{k+1}(c)$ for all $k\geq 0$. In order to prove the claim, we proceed through the following steps.
 \begin{enumerate}
  \item[(i)]  We may assume that $I$ is a {\it long} atom of $\mathcal{P}_n(c)$, say $I=f^{q_{n+1}-i}(I_n(c))$, where $1\leq i\leq q_{n+1}-1$. 
  If $I$ happens to be a {\it short} atom, all we have to do is recall that every short atom of $\mathcal{P}_n(c)$ is a long atom of $\mathcal{P}_{n+1}(c)$. 
  \item[(ii)] The interval $T=f^{q_{n+1}}(I_n(c))$ contains the interval $J_{n+4}$ in its interior, with definite space on both sides (see Figure \ref{fig2}). To see why this is true, first note that, by the real bounds, the interval $J_{n+4}$ is comparable to $|I_n(c)|$, {\it i.e.,} $|J_{n+4}|\asymp |I_n(c)|$. Consider the following two atoms of $\mathcal{P}_{n+1}(c)$, which also lie inside $T$:
  \[
   L^*= f^{q_{n+1}}(I_{n+2})\subset I_{n+1}(c)\ \ \ \textrm{and}\ \ \ 
   R^*= f^{q_n+q_{n+1}}(I_{n+1}(c))\subset I_n(c)\ .
   \]
 Both these intervals share an endpoint with $T$ (one on the left, the other on the right). 
 By simple combinatorics, we see that $J_{n+4}\subset T$ is disjoint from both $L^*$ and $R^*$. But by the real bounds, we have $|L^*|\asymp |I_{n+1}(c)|$ and $|R^*|\asymp |I_n(c)|$. If we denote by $L$ and $R$ the two connected components of $T\setminus J_{n+4}$, then one of them contains $L^*$ and the other contains $R^*$. For definiteness, we assume that 
 $L\supseteq L^*$ and $R^*\supseteq R$. Hence we have $|L|\asymp |I_{n+1}(c)|\asymp |T|$ and $|R^*|\asymp |I_n(c)|\asymp |T|$. 
 \item[(iii)] In particular, (ii) tells us that the cross-ratio $\bm{[}J_{n+4},f^{q_{n+1}}(I_n(c))\bm{]}$ is bounded away from $0$ and $\infty$.
 \item[(iv)] Now look at the interval 
 \[
 f^{-i}(J_{n+4})\subset f^{-i}(f^{q_{n+1}}(I_n(c)))=f^{q_{n+1}-i}(I_n(c))=I\ .
 \]
 Observe that $f^{-i}(J_{n+4})= I_{n+4}(c_{-i})\cup I_{n+5}(c_{-i})$ (in the notation introduced prior to Lemma \ref{cruciallemma}). Hence we can apply Lemma \ref{cruciallemma} (with $n$ replaced by $n+4$) and deduce that there exist intervals 
 \[
 \Delta'=\Delta_{i,n+4}'\subset I_{n+5}(c_{-i}) \ \ \ \text{and}\ \ \  
 \Delta''= \Delta_{i,n+4}''\subset I_{n+4}(c_{-i})
 \]
 satisfying properties (i)-(iv) of that lemma. In particular, we have 
 \begin{equation}\label{thetwodeltas}
 |\Delta'|\asymp |f^{-i}(J_{n+4})|\asymp |\Delta''|\ .
 \end{equation} 
 \item[(v)] The intervals $\Delta'$ and $\Delta''$ already satisfy properties (a) and (c) in the claim. Therefore, all we have to do is to verify that (b) holds as well. For this, it suffices to show that the intervals $f^{-i}(J_{n+4})$ and $I=f^{q_{n+1}}(I_n(c_{-i}))$ have comparable lengths. Let $L_i=f^{-i}(L)$ and $R_i=f^{-i}(R)$ be the two connected components of $I\setminus f^{-i}(J_{n+4})$. Since $L_i\supset f^{-i}(L^*)$ and $R_i\supset f^{-i}(R^*)$, and since 
 \[
  f^{-i}(L^*)=f^{q_{n+1}-i}(I_{n+2})\ \ \ \textrm{and}\ \ \ 
  f^{-i}(R^*)= f^{q_n+q_{n+1}-i}(I_{n+1}(c))
 \]
are both atoms of $\mathcal{P}_{n+1}(c)$ contained in the same atom $I\in \mathcal{P}_n(c)$, we deduce from the real bounds that $|L_i|\asymp|I|\asymp |R_i|$. By the cross-ratio inequality, the cross-ratio distortion 
$\mathrm{CrD}(f^i; f^{-i}(J_{n+4}), I)$ is bounded above. Combining this fact with (iii), we deduce that the cross-ratio $[f^{-i}(J_{n+4}), I]$ is bounded below. Since the two lateral intervals $L_i,R_i\subset I$ and the total interval $I$ have comparable lengths, it follows that the middle interval $f^{-i}(J_{n+4})\subset I$ also has length comparable to $|I|$. 
Together with \eqref{thetwodeltas}, this shows at last that 
$|\Delta'|\asymp |I|\asymp |\Delta''|$. 
 \end{enumerate}
 
This completes the proof of our claim. And as we had already observed, the claim implies that $f$ satisfies the hypotheses of Theorem \ref{katzcrit}.
Therefore it satisfies the conclusion as well: $f$ does not admit a $\sigma$-finite absolutely continuous invariant measure. This finishes the proof of Theorem \ref{nosigbound}.
\end{proof}

\subsection{The punchline} Our main theorem, namely Theorem \ref{main}, is now an immediate consequence of steps 1 and 2, or more precisely, of Theorems \ref{nosigcrit} and \ref{nosigbound}. 

\appendix

\section{The negative Schwarzian property\\ for bounded combinatorics}\label{app}

Our goal in this appendix is to provide a proof of both Proposition \ref{expdecay} and Proposition \ref{propnegschw}.

\subsection{Bounded geometry}\label{secapbounds} Let $f$ be a $C^3$ multicritical circle map (as in Definition \ref{defmccm}) with irrational rotation number. We say that $f$ has \emph{bounded geometry} at $x \in S^1$ if there exists $K>1$ such that for all $n\in\nt$ and for every pair $I,J$ of adjacent atoms of $\mathcal{P}_n(x)$ we have$$K^{-1}\,|I| \leq |J| \leq K\,|I|\,.$$Following \cite[Section 1.4]{dFG19}, we consider the set$$\mathcal{A}=\mathcal{A}(f)=\{x \in S^1:\,\mbox{$f$ has bounded geometry at $x$}\}\,.$$In other words, $x\in\mathcal{A}$ if $|I|\asymp|J|$ for any two adjacent atoms $I$ and $J$ of the dynamical partition associated to $x$ at any level $n$. As explained in \cite[Section 1.4]{dFG19}, the set $\mathcal{A}$ is $f$-invariant. Moreover, as it follows from the classical real bounds of Herman and \'Swi\c{a}tek (Theorem \ref{realbounds}), all critical points of $f$ belong to $\mathcal{A}$. Being $f$-invariant and non-empty, the set $\mathcal{A}$ is dense in the unit circle. However, even in the case of maps with a single critical point, $\mathcal{A}$ can be rather small. Indeed, the following is \cite[Theorem D]{dFG19}.

\begin{theorem}\label{teoAmagro} There exists a full Lebesgue measure set $\bm{R}\subset(0,1)$ of irrational numbers with the following property: let $f$ be a $C^3$ critical circle map with a single (non-flat) critical point and rotation number $\rho\in\bm{R}$. Then the set $\mathcal{A}(f)$ is meagre (in the sense of Baire) and it has zero $\mu$-measure (where $\mu$ denotes the unique $f$-invariant probability measure).
\end{theorem}

By contrast, if $f$ has bounded combinatorics, then the set $\mathcal{A}(f)$ is the whole circle (as a consequence, the full Lebesgue measure set $\bm{R}\subset(0,1)$ given by Theorem \ref{teoAmagro} contains no bounded type numbers). Let us be more precise.

\begin{theorem}\label{realboundsBT} For any given multicritical circle map $f$ with bounded combinatorics there exists a constant $C>1$, depending only on $f$, such that for any given point $x \in S^1$, for all $n\in\nt$, and for every pair $I,J$ of adjacent atoms of $\mathcal{P}_n(x)$ we have:$$C^{-1}\,|I| \leq |J| \leq C\,|I|\,.$$
\end{theorem}

We remark that, precisely because $f$ has bounded combinatorics, Proposition \ref{expdecay} follows at once from Theorem \ref{realboundsBT}. As explained in \cite[Section 1.3]{dFG19}, Theorem \ref{realboundsBT} follows from a result of Herman \cite{H}, which states that $f$ is \emph{quasisymmetrically} conjugate to the corresponding rigid rotation. For the sake of completeness (and because it is going to be crucial in Section \ref{secapschw} below), we would like to end Section \ref{secapbounds} by providing a different proof of Theorem \ref{realboundsBT}, without using Herman's result. With this purpose, we state first the following immediate consequence of Theorem \ref{realbounds}, which only holds for bounded combinatorics (if $\rho(f)=[a_0,a_1,...]$ with $\sup_{n\in\nt}\{a_n\} \leq B$, we say that $f$ has \emph{combinatorics bounded by $B$}).

\begin{coro}\label{coroRBBT} Given $B>1$, $N\geq 1$ in $\nt$ and $d>1$ there exists $C=C(B,N,d)>1$ with the following property: for any given multicritical circle map $f$ with combinatorics bounded by $B$, and with at most $N$ critical points whose criticalities are bounded by $d$, there exists $n_0=n_0(f)\in\nt$ such that for each critical point $c$ of $f$, for all $n \geq n_0$ and for every pair of intervals $I\in\mathcal{P}_{n}(c)$ and $J\in\mathcal{P}_{n+1}(c)$ satisfying $J\subseteq I$, we have that\, $|I| \leq C\,|J|$.
\end{coro}

The next result we will prove states that any two intersecting atoms belonging to the same level $n$ of the dynamical partitions associated to a critical and a regular point respectively, are comparable. Both its statement and its proof are essentially borrowed from \cite[Lemma 4.1, page 5599]{EdF18}.

\begin{lemma}\label{intersectcomp} Let $f$ be a multicritical circle map with bounded combinatorics. Let $c$ be a critical point of $f$, and let $x_0$ be any point in the circle. If $\Delta\in \mathcal{P}_n(c)$ and $\Delta'\in \mathcal{P}_n(x_0)$ are two atoms 
such that $\Delta\cap \Delta'\neq \textrm{\O}$, then $|\Delta|\asymp |\Delta'|$.
\end{lemma}

Note that Theorem \ref{realboundsBT} follows at once by combining Theorem \ref{realbounds} with Lemma \ref{intersectcomp} (we remark that Lemma \ref{intersectcomp} will also be used in the proof of Proposition \ref{lemmanegsch} below). During the proof of Lemma \ref{intersectcomp} we will use the following fact, which is \cite[Lemma 3.3, page 5593]{EdF18}.

\begin{lemma}\label{lema33} There exists a constant  $C> 1$, depending only on $f$, such that for all\, $n\geq 0$ and all\, $x \in S^{1}$ we have:$$C^{-1}\,\big|x- f^{-q_{n}}(x)\big|\leq\big|f^{q_{n}}(x)-x\big|\leq C\,\big|x- f^{-q_{n}}(x)\big|\,.$$
\end{lemma}

\begin{proof}[Proof of Lemma \ref{intersectcomp}] There are three cases to consider, according to the types of atoms we have: long/long, long/short, and short/short. More precisely, we have the following three cases.
\begin{enumerate}
\item[(i)] We have $\Delta=I_n^i(c)$ and $\Delta'=I_n^j(x_0)$, where $0\leq i,j <q_{n+1}$. 
Here we may assume that $f^j(x_0)\in \Delta= [f^i(c),f^{i+q_n}(c)]$, and then $f^{i+q_n}(c)\in \Delta'=[f^j(x_0),f^{j+q_n}(x_0)]$. Using the monotonicity of $f^{q_n}$, we see that $\Delta'\subset \Delta\cup f^{q_n}(\Delta)$. Applying Lemma \ref{lema33} to $x=f^{i+q_n}(c)$, we see that $\Delta=[f^{-q_n}(x),x]$ and $f^{q_n}(\Delta)=[x, f^{q_n}(x)]$ satisfy $|f^{q_n}(\Delta)|\leq C|\Delta|$, and from this it follows that $|\Delta'|\leq (1+C)|\Delta|$. Conversely, we also have $\Delta\subset f^{-q_n}(\Delta')\cup \Delta'$. Again applying Lemma \ref{lema33}, this time to $x=f^j(x_0)$, we deduce just as before that $|f^{-q_n}(\Delta')|\leq C|\Delta'|$, and therefore $|\Delta|\leq (1+C)|\Delta'|$. Hence $\Delta$ and $\Delta'$ are comparable in this case.
\item[(ii)] We have $\Delta=I_n^i(c)$ and $\Delta'=I_{n+1}^j(x_0)$, where $0\leq i <q_{n+1}$ and $0\leq j<q_n$. If $f^j(x_0) \in I_{n+2}^i(c)$, then $\Delta'$ intersects the interval $I_{n+1}^i(c)$, and since $\Delta'$ and $I_{n+1}^i(c)$ are long intervals of the partitions $\mathcal{P}_{n+1}(x_0)$ and $\mathcal{P}_{n+1}(c)$ respectively, case (i) above tells us that $|\Delta'|\asymp\big|I_{n+1}^{i}(c)\big|$. Moreover, by Theorem \ref{realbounds} we have $\big|I_{n+1}^{i}(c)\big|\asymp|\Delta|$ and then we deduce that $\Delta'$ is comparable to $\Delta$ in this sub-case. On the other hand, if $f^j(x_0) \notin I_{n+2}^i(c)$, then there exists $\ell\in\{0,1,...,a_{n+1}-1\}$ such that $\Delta'\cap I_{n+1}^{\ell\,q_{n+1}+i+q_n}(c) \neq \textrm{\O}$. Since $I_{n+1}^{\ell\,q_{n+1}+i+q_n}(c)$ is an atom of $\mathcal{P}_{n+1}(c)$ contained in $\Delta\in\mathcal{P}_{n}(c)$, we have by Corollary \ref{coroRBBT} that $\big|I_{n+1}^{\ell\,q_{n+1}+i+q_n}(c)\big|\asymp|\Delta|$. Now, since $\Delta'$ also belongs to $\mathcal{P}_{n+1}(x_0)$, case (i) above tells us that $|\Delta'|\asymp\big|I_{n+1}^{\ell\,q_{n+1}+i+q_n}(c)\big|$, and therefore $\Delta'$ is comparable to $\Delta$ in this sub-case too.
\item[(iii)] We have $\Delta=I_{n+1}^i(c)$ and $\Delta'=I_{n+1}^j(x_0)$, where $0\leq i,j <q_{n}$. 
This case is entirely analogous to case (i).
\end{enumerate}
\end{proof}

\subsection{The negative Schwarzian property}\label{secapschw} The remainder of this appendix is devoted to establish the following two facts.

\begin{prop}[The $C^1$ bounds]\label{lemmaC1bounds} For any given multicritical circle map $f$ with bounded combinatorics there exists a constant $K=K(f)>1$ such that the following holds. For any given $x_0 \in S^1$ and $n\in\nt$ let\, $I_n=I_n(x_0)$. Then $Df^k(x)\leq K\,\dfrac{\big|f^{k}(I_n)\big|}{|I_n|}$ for all $x \in I_n$ and all $k\in\{0,1,...,q_{n+1}\}$. Moreover $\big\|f^{q_{n+1}}\big\|_{C^1(I_n)} \leq K$. Likewise, if\, $I_{n+1}=I_{n+1}(x_0)$, then $Df^k(x)\leq K\,\dfrac{\big|f^{k}(I_{n+1})\big|}{|I_{n+1}|}$ for all $x \in I_{n+1}$ and all $k\in\{0,1,...,q_{n}\}$, and moreover $\big\|f^{q_{n}}\big\|_{C^1(I_{n+1})} \leq K$.
\end{prop}

\begin{prop}[The negative Schwarzian property]\label{lemmanegsch} For any given multicritical circle map $f$ with bounded combinatorics there exists $n_0=n_0(f)\in\nt$ such that for all $x_0 \in S^1$ and all $n \geq n_0$ we have\[ Sf^{q_{n+1}}(x)<0\quad\text{for all $x \in I_{n}(x_0)$ regular point of $f^{q_{n+1}}$.}
 \]
Likewise, we have 
\[ Sf^{q_{n}}(x)<0\quad\text{for all $x \in I_{n+1}(x_0)$ regular point of $f^{q_{n}}$}.
\]
\end{prop}

Note that Proposition \ref{lemmanegsch} immediately implies Proposition \ref{propnegschw}. We remark that both Proposition \ref{lemmaC1bounds} and Proposition \ref{lemmanegsch} are well known in the case when $x_0$ is a critical point of $f$, in which case they hold true for any irrational rotation number: see \cite[Appendix A]{dFdM99} for the case of critical circle maps with a single critical point, and see \cite[Sections 3 and 4]{EdFG18} for the case of multicritical circle maps. Our goal in this appendix is to generalize both results to the case when $x_0$ is a regular point of a multicritical circle map with bounded combinatorics. In the proof of Proposition \ref{lemmaC1bounds} we adapt the exposition in \cite[pages 849-851]{EdFG18}, whereas in the proof of Proposition \ref{lemmanegsch} we adapt the exposition in \cite[pages 380-381]{dFdM99}.

\begin{proof}[Proof of Proposition \ref{lemmaC1bounds}] We give the proof only for the case $x \in I_n(x_0)$ (the other case being entirely analogous). We consider the three intervals $L_{n}=I_{n+1}(x_0)$, $R_{n}= f^{q_{n}}\big(I_n(x_0)\big)$ and $T_{n}=L_{n} \cup I_{n}(x_0) \cup R_{n}$. From now on we suppress the point $x_0$ to simplify the notation. We need to establish two preliminary facts.

\begin{lemma}\label{lemaspaceapp} There exists a constant $\tau>0$ (depending only on $f$) such that$$|L_n^j|>\tau|I_n^j|\quad\mbox{and}\quad|R_n^j|>\tau|I_n^j|$$for each $j \in \{ 0, \cdots, q_{n+1} \}$ and for all $n\in\nt$.
\end{lemma}

\begin{proof}[Proof of Lemma \ref{lemaspaceapp}] For $j=0$, observe that the intervals $L_n$, $I_n$ and $R_n$ are adjacent and belong to the dynamical partition $\mathcal{P}_n$, then by Theorem \ref{realboundsBT} they are comparable. Let us prove now that for $j= q_{n+1}$ the three intervals $L_n^{j}$, $I_n^j$ and $R_n^j$ are comparable too. On one hand, the intervals $I_{n+1}$ and $I_{n+1}^{q_{n+1}}$ are adjacent and belong to $\mathcal{P}_{n+1}$, then they are comparable (again by Theorem \ref{realboundsBT}). Moreover $I_{n+1}\subset I_{n}^{q_{n+1}} \subset I_{n+1}\cup I_n$. By Theorem \ref{realboundsBT}, $|I_n|\asymp|I_{n+1}|$ and then $|I_{n}^{q_{n+1}}|\asymp|I_{n+1}^{q_{n+1}}|$, that is:
\begin{equation} \label{fact2}
|L_{n}^{q_{n+1}}| \asymp |I_{n}^{q_{n+1}}|\,.
\end{equation}On the other hand, the intervals $I_{n}$ and $I_{n}^{q_{n}}$ are adjacent and belong to $\mathcal{P}_{n}$, then they are comparable. Moreover:$$I_{n+1}^{q_{n}}\subset I_{n}^{q_n+q_{n+1}}\subset I_n \cup I_{n}^{q_{n}}\,.$$By Corollary \ref{coroRBBT}, we know that $|I_{n+1}^{q_{n}}| \asymp |I_{n}|$ and then $|I_{n}^{q_{n}+q_{n+1}}| \asymp |I_{n}|$. Using again that $I_{n+1}\subset I_n^{q_{n+1}}\subset I_n\cup I_{n+1}$, we conclude from Theorem \ref{realboundsBT} that:
\begin{equation} \label{fact22}
 |R_{n}^{q_{n+1}}|=|I_{n}^{q_{n}+q_{n+1}}| \asymp |I_{n}| \asymp |I_n^{q_{n+1}}|\,.
\end{equation}Therefore, for $j= q_{n+1}$, the three intervals $L_n^{j}$, $I_n^j$ and $R_n^j$ are comparable. Now, let $1 \leq j \leq q_{n+1}-1$. Consider the intervals $|L_{n}^j|, \ |I_{n}^j| , \ |R_{n}^j|$ and their images by the map $f^{q_{n+1}-j}$. Since, by combinatorics, the family $\{ T_{n}, f(T_{n}), \cdots, f^{q_{n+1}-1}(T_{n}) \}$ has intersection multiplicity bounded by $3$, the Cross-Ratio Inequality (see Section \ref{secprel}) implies that there exists a constant $K_0>1$, depending only on $f$, such that
\begin{equation*}
  \dfrac{|L_{n}^{q_{n+1}}|\,|R_{n}^{q_{n+1}}|\,|L_{n}^{j} \cup I_{n}^{j}|\,|I_{n}^{j} \cup R_{n}^{j}|}{|L_{n}^{j}|\, |R_{n}^{j}|\,|L_{n}^{q_{n+1}} \cup I_{n}^{q_{n+1}}|\,|I_{n}^{q_{n+1}} \cup R_{n}^{q_{n+1}}|} \leq K_0\,.
\end{equation*}Using \eqref{fact2} and \eqref{fact22} in the last inequality, we get
\begin{equation*}
 \left( 1 + \dfrac{|I_{n}^j|}{|L_{n}^{j}|} \right) \left( 1 + \dfrac{|I_{n}^j|}{|R_{n}^j|} \right) \leq K\,,
\end{equation*}
and we are done.
\end{proof}

We assume from now on that $n\in\nt$ is large enough so that for all $j\in\{0,...,q_{n+1}\}$ we have $\car( f^{j}(T_{n}) \cap\crit(f)) \leq 1$, where $\car$ denotes the cardinality of a finite set, and $\crit(f)$ is the set of critical points of $f$ (recall that, by minimality, $\big|f^j(T_n)\big|$ goes to zero as $n$ goes to infinity). We say that $j \in \{1, \cdots, q_{n+1} \}$ is a \emph{critical time} if $f^{j}(T_{n}) \cap\crit(f) \neq \textrm{\O}$. Note that $\car( \{\mbox{critical times} \} )\leq 3N$, by combinatorics.

\begin{lemma}\label{regbranches} Let $1 \leq j_1<j_2\leq q_{n+1}$ be two consecutive critical times. Then for all $x \in f^{j_1+1}(I_n)$ we have$$Df^{j_2-j_1-1}(x) \asymp \frac{|f^{j_2}(I_n)|}{|f^{j_1+1}(I_n)|}\,.$$
\end{lemma}

\begin{proof}[Proof of Lemma \ref{regbranches}] Note that $f^{j_2-j_1-1}:f^{j_1+1}(T_n) \to f^{j_2}(T_n)$ is a diffeomorphism. Using again that the family $\{ T_{n}, f(T_{n}), \cdots, f^{q_{n+1}-1}(T_{n}) \}$ has intersection multiplicity bounded by $3$, we have $\sum_{i=0}^{j_2-j_1-1}|f^{i}(f^{j_1+1}(T_n))| < 3$. Moreover, by Lemma \ref{lemaspaceapp}, the interval$$f^{j_2-j_1-1}(f^{j_1+1}(T_n))$$contains a $\tau-$scaled neighbourhood of $f^{j_2-j_1-1}(f^{j_1+1}(I_n))$. Therefore, by Koebe's Distortion Principle (see Section \ref{secprel}) there exists a constant $K_0=K_0(f)>1$ such that for all $x,y \in f^{j_1+1}(I_n)$ we have that
\begin{equation*}
 \dfrac{1}{K_0} \leq \dfrac{Df^{j_2-j_1-1}(x)}{Df^{j_2-j_1-1}(y)} \leq K_0\,.
\end{equation*}
Let $y\in I_n^{j_1+1}$ be given by the Mean Value Theorem:$$Df^{j_2-j_1-1}(y)=\dfrac{|f^{j_2}(I_n)|}{|f^{j_1+1}(I_n)|}\,.$$Then for all $x \in f^{j_1+1}(I_n)$ we have
\begin{equation*}
 \dfrac{1}{K_0}\,\dfrac{|f^{j_2}(I_n)|}{|f^{j_1+1}(I_n)|} \leq Df^{j_2-j_1-1}(x) \leq K_0\,\dfrac{|f^{j_2}(I_{n})|}{|f^{j_1+1}(I_n)|}\,. 
\end{equation*}
\end{proof}

Now recall that, by the non-flatness condition (see for instance \cite[Lemma 2.2]{EdFG18}), for each critical point $c_{i}$ of criticality $d_i>1$ there exists a neighbourhood $U_{i} \subseteq S^{1}$ of $c_i$ such that for any given interval $J \subseteq U_i$ and $x \in J$ we have$$Df(x) \leq 3d_i\,\dfrac{\big|f(J)\big|}{|J|}\,.$$With this at hand, the first estimate in Proposition \ref{lemmaC1bounds} follows from Lemma \ref{regbranches} and the help of the 
chain rule: $$Df^{j}(x) \leq (3d)^{3N}\,K_0^{3N}\,\dfrac{|f^{j}(I_{n})|}{|I_{n}|}\quad\mbox{for any $x \in I_{n}$ and 
$j \in \{1, \cdots, q_{n+1} \}$\,,}$$where $N=\car\big(\crit(f)\big)$ is the number of critical points of $f$, $d$ is the maximum of its criticalities and $K_0=K_0(f)$ is given by Lemma \ref{regbranches}. We finish the proof of Proposition \ref{lemmaC1bounds} by proving that the sequence $\big\{f^{q_{n+1}}|I_{n}\big\}$ is bounded in the $C^1$ metric: by combinatorics, $I_{n+1} \subset f^{q_{n+1}}(I_n) \subset I_n \cup I_{n+1}$. Then:$$\frac{|I_{n+1}|}{|I_{n}|}\leq\frac{\big|f^{q_{n+1}}(I_n)\big|}{|I_{n}|}\leq 1+\frac{|I_{n+1}|}{|I_{n}|}\,.$$By Theorem \ref{realboundsBT}, we have $|I_{n+1}|\asymp|I_n|$, and then $\big|f^{q_{n+1}}(I_{n})\big|\asymp|I_{n}|$.
\end{proof}

With Lemma \ref{intersectcomp} and Proposition \ref{lemmaC1bounds} at hand, we are ready to prove our main result in this appendix, namely Proposition \ref{lemmanegsch} (and remember that Proposition \ref{lemmanegsch} immediately implies Proposition \ref{propnegschw}).

\begin{proof}[Proof of Proposition \ref{lemmanegsch}] Let us fix $x_0 \in S^1$ and $n\in\nt$. We give the proof only for the case $x \in I_n(x_0)$ regular point of $f^{q_{n+1}}$ (the other case being entirely analogous). Let $j\in\{0,...,q_{n+1}-1\}$ be the minimum positive integer such that$$f^{j}\big(I_{n}(x_0)\big) \cap J_{n}(c_i)\neq \textrm{\O}$$for some $i\in\{0,...,N-1\}$. Without loss of generality, we may assume that $i=0$. By Lemma \ref{intersectcomp}, $f^{j}(I_{n}(x_0))$ and $J_{n}(c_0)$ have comparable lengths: there exists $C_0>1$, depending only on $f$, such that$$\big|f^{j}(x)-c_0\big| \leq C_0\,\big|f^{j}(I_{n}(x_0))\big|\quad\mbox{for all $x \in I_{n}(x_0)$.}$$
Moreover, by Koebe distortion principle there exists $C_1>1$ (also depending only on $f$) such that $f^j|_{I_{n}(x_0)}$ has distortion bounded by $C_1$, that is:$$\frac{1}{C_1}\leq\frac{Df^j(x)}{Df^j(y)}\leq C_1\quad\mbox{for all $x,y \in I_{n}(x_0)$.}$$Recall that, by the non-flatness condition (see for instance \cite[Lemma 2.2]{EdFG18}), for each critical point $c_{i}$ there exist a neighbourhood 
$U_{i} \subseteq S^{1}$ of $c_i$ and a positive
constant $K_i$ such that for all $x \in U_{i} \setminus \{c_{i} \}$ we have
\begin{equation}\label{negS0}
 Sf(x)< -\,\dfrac{K_{i}}{(x-c_{i})^{2}} <0\,.
\end{equation}
Let $\mathcal{U}=\bigcup_{i=0}^{N-1} U_i$, and let $\mathcal{V}\subset S^1$ be an open set whose closure contains no critical point of $f$ and such that $\mathcal{U}\cup \mathcal{V}=S^1$. Since $f$ is of class $C^3$, we know that $M=\sup_{y \in \mathcal{V}}\big|Sf(y)\big|$ is finite. Let $\delta_n=\max_{x_0 \in S^1}\max_{0\leq k< q_{n+1}}\big|f^k(I_{n}(x_0))\big|$.  Since $f$ is minimal, $\delta_n\to 0$ as $n\to \infty$. 
We choose $n_0=n_0(f)$ so large that $\delta_n$ is smaller than the Lebesgue number of the covering $\{\mathcal{U}, \mathcal{V}\}$ of the circle for all $n\geq n_0$. Moreover, we also require that $\delta_n<K_0/M\,K^2\,C_0^2\,C_1^2$ for all $n\geq n_0$, where $K=K(f)>1$ is given by Proposition \ref{lemmaC1bounds}. Using the chain rule for the Schwarzian derivative, we have for all $\ell\in \{j+1,...,q_{n+1}\}$ and for all $x \in I_{n}(x_0)$ regular point of $f^{\ell}$ the following identity:
$$Sf^{\ell}(x)\;=\; \sum_{k=0}^{\ell-1} Sf(f^k(x))\left[Df^k(x)\right]^2.$$
We decompose this expression as $\Sigma_1^{(n)}(x) + \Sigma_2^{(n)}(x)$, where
 \begin{equation}\label{S1}
 \Sigma_1^{(n)}(x)\;=\; \sum_{k:f^k(I_{n}(x_0)) \subset \mathcal{U}} Sf(f^k (x))\left[Df^k (x)\right]^2 \ ,
\end{equation}
and $\Sigma_2^{(n)}(x)$ is the sum over the remaining terms, and we treat both cases separately.
\begin{enumerate}
 \item[(i)] Since $f^{j}(I_{n}(x_0)) \cap J_n(c_0)\neq\text{\O}$, we have $f^{j}(I_{n}(x_0))\subset\mathcal{U}$ and then the sum in the right-hand side of \eqref{S1} includes the term with $k=j$, namely $Sf\big(f^j(x)\big)\left[Df^j(x)\right]^2$. Since all the other terms in \eqref{S1} are negative as well, and since $\big|f^{j}(x)-c_0\big| \leq C_0\,\big|f^{j}(I_{n}(x_0))\big|$, we deduce from \eqref{negS0} that:$$\Sigma_1^{(n)}(x)\;<\; -\,\frac{K_0}{C_0^2\,\big|f^{j}(I_{n}(x_0))\big|^2}\,\left[Df^j(x)\right]^2\,.$$
Let $y \in I_{n}(x_0)$ be such that $\big|f^{j}(I_{n}(x_0))\big|=Df^j(y)\,|I_{n}(x_0)|$. By bounded distortion, we obtain:
\begin{equation}\label{S1final}
\Sigma_1^{(n)}(x)\;<\; -\,\frac{K_0}{C_0^2}\,\frac{1}{|I_{n}(x_0)|^2}\left[\frac{Df^j(x)}{Df^j(y)}\right]^2\;<\; -\,\frac{K_0}{C_0^2\,C_1^2}\,\frac{1}{|I_{n}(x_0)|^2}\,.
\end{equation}
\item[(ii)] Observe that
$$\left|\Sigma_2^{(n)}(x)\right|\;\leq\; \sum_{k:f^k(I_{n}(x_0))\subset \mathcal{V}}\big|Sf(f^k(x))\big|\left[Df^k(x)\right]^2.$$
By Proposition \ref{lemmaC1bounds}, there exists $K>1$ such that
\begin{equation}\label{S2final}
\begin{aligned}
 \left|\Sigma_2^{(n)}(x)\right| &\leq \sum_{k:f^k(I_{n}(x_0))\subset \mathcal{V}}\big|Sf(f^k(x))\big|\,K^{2}\, \dfrac{|f^k(I_{n}(x_0))|^2}{|I_{n}(x_0)|^2} \\
 &\leq M\,\dfrac{K^2}{|I_{n}(x_0)|^2}\sum_{k:f^k(I_{n}(x_0))\subset \mathcal{V}}\big|f^k(I_{n}(x_0))\big|^2 \\
 &\leq  M\,\dfrac{K^2}{|I_{n}(x_0)|^2}\,\max_{0 \leq k \leq \ell-1}\big|f^k(I_{n}(x_0))\big|\sum_{k:f^k(I_{n}(x_0))\subset \mathcal{V}}\big|f^k(I_{n}(x_0))\big|\\
&\leq  M\,\dfrac{K^2}{|I_{n}(x_0)|^2}\,\delta_{n}.
\end{aligned}
\end{equation}
\end{enumerate}
By our choice of $n_0$, we know that $\delta_n<K_0/M\,K^2\,C_0^2\,C_1^2$ for all $n\geq n_0$, and then we deduce from \eqref{S1final} and \eqref{S2final} that, indeed, $Sf^{\ell}(x)<0$ for all $\ell \in \{j+1, \cdots,q_{n+1}\}$ and all $x \in I_n(x_0)$ regular point of $f^{\ell}$.
\end{proof}

\end{document}